\documentclass[10pt, reqno]{amsart}
\usepackage{amsfonts}
\usepackage{mathrsfs}
\usepackage{amscd}
\usepackage{amsmath}
\usepackage{amssymb}
\usepackage{latexsym}
\usepackage{lscape}
\usepackage{xypic}
\usepackage{comment}
\usepackage{amscd}
\usepackage{wasysym}  
\usepackage{tikz} 
\usetikzlibrary{matrix,arrows}
\usepackage{tikz-cd}
\usepackage{appendix}
\usepackage{geometry}
\usepackage[nice]{nicefrac}
\usepackage{dsfont}
\usepackage{color}
\usepackage{graphicx}
\usepackage{todonotes}
\usepackage{enumerate}
\usepackage{stmaryrd}
\usepackage{mathtools}

\usepackage{theoremref}
\usepackage[english]{babel}
\usepackage{bbm}

\usepackage[hyperindex,linktocpage=true]{hyperref}
\hypersetup{
	colorlinks,
	linkcolor={red!50!black},
	citecolor={blue!50!black},
	urlcolor={blue!80!black}
}

\newtheorem{thm}{Theorem}[section]

\newtheorem{prop}[thm]{Proposition}
\newtheorem{lem}[thm]{Lemma}

\newtheorem{lem-def}[thm]{Lemma-Definition}
\newtheorem{cor}[thm]{Corollary}

\theoremstyle{definition}

\newtheorem{rmk}[thm]{Remark}
\newtheorem{dfn}[thm]{Definition}

\newcommand{\Dd}{\mathcal{D}}

\newcommand{\Gg}{\mathcal{G}}
\newcommand{\Hh}{\mathcal{H}}
\newcommand{\Ii}{\mathcal{I}}

\newcommand{\Mm}{\mathcal{M}}

\newcommand{\Oo}{\mathcal{O}}
\newcommand{\Pp}{\mathcal{P}}

\newcommand{\Tt}{\mathcal{T}}

\newcommand{\Zz}{\mathcal{Z}}

\newcommand{\IA}{\mathbb{A}}

\newcommand{\IF}{\mathbb{F}}
\newcommand{\IG}{\mathbb{G}}

\newcommand{\IQ}{\mathbb{Q}}

\newcommand{\IZ}{\mathbb{Z}}

\renewcommand{\aa}{\mathfrak{a}}

\numberwithin{equation}{section}

\DeclareMathOperator{\End}{End} 
\DeclareMathOperator{\Gr}{Gr} 
\DeclareMathOperator{\Spec}{Spec} 
\DeclareMathOperator{\identity}{id} 
\DeclareMathOperator{\redu}{red} 
\DeclareMathOperator{\Gal}{Gal} 
\DeclareMathOperator{\der}{der} 
\DeclareMathOperator{\adj}{ad} 
\DeclareMathOperator{\Aut}{Aut} 
\DeclareMathOperator{\Ad}{Ad} 
\DeclareMathOperator{\Loc}{Loc} 
\DeclareMathOperator{\Sh}{Sh} 

\newcommand{\into}{\hookrightarrow}

\newcommand{\GL}{\mathrm{GL}} 
\newcommand{\PGL}{\mathrm{PGL}} 
\newcommand{\CG}{{}^cG} 
\renewcommand{\phi}{\varphi}
\newcommand{\cInd}{\operatorname{c}\text{-}\operatorname{Ind}} 
\newcommand{\sph}{\mathrm{sph}} 

\newcommand{\vp}{\varphi}
\newcommand{\bZ}{{\mathbf Z}}

\DeclareMathOperator{\Lie}{Lie}

\begin{document}

\title[The stack of spherical Langlands parameters]{The stack of spherical Langlands parameters}
\author[Thibaud van den Hove]{Thibaud van den Hove\\with an appendix by Sean Cotner}
	
\address{Max Planck Institute for Mathematics, Vivatsgasse 7, 53111 Bonn, Germany}
\email{vandenhove@mpim-bonn.mpg.de}

\begin{abstract}
	For a reductive group over a nonarchimedean local field, we define the stack of spherical Langlands parameters, using the inertia-invariants of the Langlands dual group.
	This generalizes the stack of unramified Langlands parameters in case the group is unramified.
	We then use this stack to deduce the Eichler--Shimura congruence relations for Hodge type Shimura varieties, without restrictions on the ramification.
\end{abstract}

\maketitle

\setcounter{tocdepth}{1}
\tableofcontents
\setcounter{section}{0}


\thispagestyle{empty}

\section{Introduction}

For a reductive group \(G\) over a nonarchimedean local field \(F\), Langlands duality is governed by the Langlands dual group \(\widehat{G}\).
It arises naturally from the Satake isomorphism for spherical Hecke algebras \cite{Satake:Theory, Borel:Automorphic} and the geometric Satake equivalence \cite{MirkovicVilonen:Geometric}.
For unramified groups, this relates unramified representations with unramified Langlands parameters, i.e., semisimple elements of \(\widehat{G}(\mathbb{C})\) up to (twisted) \(\widehat{G}\)-conjugation.
For ramified groups, there are also Satake isomorphisms \cite{HainesRostami:Satake,vdH:RamifiedSatake} and geometric Satake equivalences \cite{Zhu:Ramified,Richarz:Affine}, where instead the inertia-invariants \(\widehat{G}^I\) appear.
However, when defining Langlands parameters, including the stacks of Langlands parameters as in \cite{DatHelmKurinczukMoss:Moduli, Zhu:Coherent, FarguesScholze:Geometrization}, one still takes the quotient by \(\widehat{G}\)-conjugation, even for ramified groups.
The aim of this paper is twofold:
\begin{enumerate}
	\item Explain how the inertia-invariants \(\widehat{G}^I\) appear in the stacks of Langlands parameters from \cite{DatHelmKurinczukMoss:Moduli, Zhu:Coherent,FarguesScholze:Geometrization}.
	\item Give an application to the cohomology of Shimura varieties, in the form of the Eichler--Shimura congruence relations when the group is not necessarily unramified at \(p\).
\end{enumerate}

Let us start by explaining the second point.

\subsection{Eichler--Shimura relations in the ramified case}

Fix a prime \(p\).
The Eichler--Shimura relations for the modular curve give
\(T_p = \operatorname{Frob} + \operatorname{Ver},\)
where \(T_p\) is a certain Hecke correspondence, and \(\operatorname{Frob}\) and \(\operatorname{Ver}\) are the Frobenius and Verschiebung correspondences on the \(\operatorname{mod} p\) reduction, where \(p\) is a prime of good reduction.
Consequently, Blasius--Rogawski \cite[§6]{BlasiusRogawski:Zeta} gave a conjectural generalization for general Shimura varieties \emph{with good reduction at \(p\)}, which is a statement on the level of cohomology.
Roughly speaking, it says that the action of Frobenius on the cohomology of a Shimura variety satisfies a certain Hecke polynomial with coefficients in the unramified Hecke algebra, defined via a representation of the Langlands dual group.
Much work has been devoted to proving these ``congruence relations'' in special cases, including \cite{FaltingsChai:Degeneration, Wedhorn:Congruence, Bueltel:Congruence, BueltelWedhorn:Congruence, Koskivirta:Congruence, Lee:EichlerShimura}, and finally Wu \cite{Wu:S=T} gave a proof for general Shimura varieties of Hodge type, at least for hyperspecial level at \(p\).
Recently, Daniels--van Hoften--Kim--Zhang \cite{DvHKZ:Igusa} constructed so-called Igusa stacks, and, combined with the spectral action from Fargues--Scholze \cite{FarguesScholze:Geometrization}, used them to give a different proof of the Blasius--Ragowski conjecture for Hodge type Shimura varieties.
In fact, they even allow Iwahori level at \(p\), as long as this Iwahori is contained in a hyperspecial subgroup.

In particular, all references above assume that the group appearing in the Shimura datum is unramified at \(p\).
A key ingredient in all the approaches above is the Satake isomorphism \cite{Satake:Theory, Borel:Automorphic}, relating the unramified Hecke algebra to rings of functions on stacks of Langlands parameters.
As mentioned before, there are versions of the Satake isomorphism for ramified groups, using the inertia-invariants \(\widehat{G}^I\) rather than \(\widehat{G}\).
Thus, it is natural to expect that in the ramified case, the action of Frobenius on the cohomology of a Shimura variety satisfies a similar Hecke polynomial as in the unramified case, but instead defined via a representation of \(\widehat{G}^I\).
This is indeed the case, and the approach from \cite{DvHKZ:Igusa} goes through with minor modifications, as soon as one can relate \(\widehat{G}^I\) to the stack of Langlands parameters constructed in \cite{DatHelmKurinczukMoss:Moduli, Zhu:Coherent, FarguesScholze:Geometrization} (at least for quasi-split groups).
See also \cite[Remark 9.4.9]{DvHKZ:Igusa}, although perhaps surprisingly, we do not need our groups to be tamely ramified.
To handle groups that are not quasi-split, we use transfer morphisms as defined by Haines \cite[§11.12]{Haines:Stable}, and compare the Fargues--Scholze local Langlands correspondence to the construction from \cite{Haines:Satake}.
We then get the following theorem, where \(E\) denotes a \(p\)-adic completion of the reflex field of the Shimura variety.

\begin{thm}
	Let \(\Sh_{K',\overline{E}}\) be a Hodge type Shimura variety with Iwahori level at \(p\), and \(\Lambda\) a \(\IZ_\ell[\sqrt{p}]\)-algebra, under a mild assumption on \(\ell\neq p\).
	Then, the inertia subgroup of \(\Gal(\overline{E}/E)\) acts unipotently on \(\mathrm{H}_{\mathrm{\acute{e}t}}^i(\Sh_{K',\overline{E}},\Lambda)\) for each \(i\), whereas the action of any lift of Frobenius satisfies an explicit Hecke polynomial.
\end{thm}

We refer to Section \ref{Sec--Congruence}, and more specifically \thref{Eichler-Shimura}, for a precise statement and details.

\subsection{The stack of spherical Langlands parameters}

Let \(G\) be a reductive group over a nonarchimedean local field \(F\) of residue characteristic \(p\).
In \cite{DatHelmKurinczukMoss:Moduli, Zhu:Coherent, FarguesScholze:Geometrization}, the various authors construct the stack of Langlands parameters for \(G\).
This stack is, roughly speaking, obtained by considering a scheme \(\Loc_{\CG}^{\square}\) parametrizing morphisms from a discretization of the Weil group to the Langlands dual group, and then quotienting out a \(\widehat{G}\)-action; we refer to loc.~cit.~for precise definitions.
If \(G\) is tamely ramified, \cite[§3.3]{Zhu:Coherent} also defines substacks of tame and unipotent parameters, which are obtained by taking a subscheme of \(\Loc_{\CG}^{\square}\) and quotienting out the restricted \(\widehat{G}\)-action.
Similarly, if \(G\) is unramified, there is the stack of unramified Langlands parameters.
This consists of those parameters which are trivial on inertia, and is isomorphic to the quotient of \(\widehat{G}\) by a suitable \(\widehat{G}\)-action.
In particular, all the stacks above are clearly (closed) substacks of \(\Loc_{\CG}\).

The stack of unramified Langlands parameters in particular is motivated by the Satake isomorphism for unramified Hecke algebras and the geometric Satake equivalence.
Thus, for ramified groups, it is natural to define the stack of \emph{spherical Langlands parameters} \(\Loc_{\CG}^{\sph}\) as a quotient of \(\widehat{G}^I\) by a suitable action of itself.
Although this makes sense in view of \cite[§6]{Zhu:Ramified} and \cite{Haines:Satake}, modding out by \(\widehat{G}^I\) makes it less clear how to relate it to \(\Loc_{\CG}\), let alone having a (closed) immersion.
Moreover, in the Satake equivalence from \cite{FarguesScholze:Geometrization}, it is also the full Langlands dual group that appears, even for ramified groups; this is exactly what allows Fargues--Scholze to construct their spectral action.
Nevertheless, it is possible to relate \(\Loc_{\CG}^{\sph}\) to \(\Loc_{\CG}\), cf.~\thref{closed immersion}.

\begin{thm}
	There is a natural closed immersion \(\Loc_{\CG}^{\sph}\to \Loc_{\CG}\).
\end{thm}

A key input for the proof, is the result that for a pinning-preserving automorphism \(\alpha\) of \(\widehat{G}\), that the natural map \(\widehat{G}/\widehat{G}^{\alpha} \to \widehat{G}\colon g\mapsto g\alpha(g^{-1})\) is a closed immersion.
This is proven in Appendix \ref{Appendix:Sean} by Sean Cotner, and allows us to show that the map \(\Loc_{\CG}^{\sph}\to \Loc_{\CG}\) is a closed immersion globally over \(\Spec \IZ[\frac{1}{p}]\), rather than only fiberwise.

\subsection{Outline}

We start by defining and studying the stack of spherical Langlands parameters in Section \ref{Sec--Stack}.
In order to handle non-quasi-split groups, we recall the transfer maps on Hecke algebras in Section \ref{Sec--Transfer}, and make a part of the Fargues--Scholze correspondence explicit in Section \ref{Sec--FS}.
We then use these results to upgrade the strategy of \cite[§9]{DvHKZ:Igusa} and prove the Eichler--Shimura relations for Hodge type Shimura varieties in Section \ref{Sec--Congruence}.
Finally, in Appendix \ref{Appendix:Sean} by Sean Cotner, the key technical result \thref{theorem:thibaud} is proven.

\subsection{Acknowledgements}

I thank Tom Haines, Timo Richarz, Pol van Hoften, Torsten Wedhorn, and Mingjia Zhang for useful discussions and comments on earlier drafts.
I also thank the authors of \cite{DvHKZ:Igusa} for sending me a preliminary version of their paper, and Tom Haines for pointing out \cite{Li:Compatibility}.
Finally, I thank Sean Cotner for writing Appendix \ref{Appendix:Sean}, allowing me to fully prove \thref{closed immersion}.

I acknowledge support by the European research council (ERC) under the European Union’s Horizon 2020 research and innovation programme (grant agreement No 101002592), by the Deutsche Forschungsgemeinschaft (DFG, German Research Foundation) TRR 326 \textit{Geometry and Arithmetic of Uniformized Structures}, project number 444845124 and the LOEWE professorship in Algebra, project number LOEWE/4b//519/05/01.002(0004)/87.

\section{The stack of spherical Langlands parameters}\label{Sec--Stack}

Let \(G\) be a reductive group over a nonarchimedean local field \(F\) of residue characteristic \(p\), and let \(\widehat{G}\) be its dual group, defined over \(\Spec \IZ[\frac{1}{p}]\).
Fix a finite extension \(\widetilde{F}/F\) splitting \(G\), and consider the C-group \(\CG:=\widehat{G}\rtimes(\IG_m\times \Gal(\widetilde{F}/F))\xrightarrow{d} \IG_m\times \Gal(\widetilde{F}/F)\) as defined in \cite[§1.1]{Zhu:Integral}.
Denote the ring of integers of \(F\) by \(\Oo\), and its residue cardinality by \(q=p^r\). 
We will also consider the Weil group, inertia, and wild inertia subgroups of \(F\), and write them as \(W_F\supset I:=I_F\supset P_F\). 
Fix a lift \(\sigma\in W_F\) of the arithmetic Frobenius, as well as a (topological) generator of the tame inertia \(\tau\in I_F/P_F\subseteq \Gamma_F^t := \Gal(\overline{F}/F)/P_F\).
This determines an injection
\[\iota \colon \langle\sigma,\tau\mid \sigma\tau\sigma^{-1}=\tau^q \rangle \hookrightarrow \Gamma_F^t.\]
Depending on this choice, Zhu \cite[§3.1]{Zhu:Coherent} defines the stack of Langlands parameters \(\Loc_{\CG} = \Loc_{\CG}^{\square}/\widehat{G}\) over \(\Spec \IZ[\frac{1}{p}]\).
For unramified reductive groups (i.e., which are quasi-split and split over an unramified field extensions), there is the substack of unramified Langlands parameters \(\Loc_{\CG}^{\mathrm{unr}} = \widehat{G}q^{-1}\sigma /\widehat{G}\) \cite[§3.3]{Zhu:Coherent}; here \(\widehat{G}q^{-1}\sigma \subseteq \widehat{G}\rtimes(\IG_m\times \Gal(\widetilde{F}/F)) = \CG\) is a closed subscheme stable for the \(\widehat{G}\)-conjugation action, and the restricted action corresponds to \(q^{-1}\sigma\)-twisted conjugation.
The goal of this section is to define an analogous substack for groups that are not necessarily unramified.
Recall the notion of spherical Langlands parameters from \cite[Definition 6.3]{Zhu:Ramified}, i.e., those parameters which are \(\widehat{G}\)-conjugate to a parameter which is trivial on inertia.
Then \cite[Lemma 6.3]{Zhu:Ramified} suggests the following definition.

\begin{dfn}
	The \emph{stack of spherical Langlands parameters} is defined as the quotient stack \(\Loc_{\CG}^{\sph} := \widehat{G}^I q^{-1}\sigma/\widehat{G}^I\), living over \(\Spec \IZ[\frac{1}{p}]\).
\end{dfn}

The above definition is useful in light of the following description of its global sections.
For a parahoric integral model \(\Gg/\Oo\) of \(G\), let \(\Hh_{\Gg}:=C_c(\Gg(\Oo)\backslash G(F)/\Gg(\Oo_F),\IZ[\frac{1}{p}])\) denote the corresponding Hecke algebra with coefficients in \(\IZ[\frac{1}{p}]\).
Recall the notion of very special parahorics \cite[Definition 6.1]{Zhu:Ramified}, which exist exactly when \(G\) is quasi-split \cite[Lemma 6.1]{Zhu:Ramified}.

\begin{prop}\thlabel{global sections of stack of spherical parameters}
	Assume \(G\) is quasi-split, and let \(\Gg/\Oo\) be a very special parahoric.
	Then, there is a canonical isomorphism
	\[\mathrm{H}^0\Gamma(\Loc_{\CG}^{\sph}, \Oo_{\Loc_{\CG}^{\sph}}) \cong \Hh_{\Gg}.\]
\end{prop}
Since each reductive group admits a unique (up to isomorphism) quasi-split inner form, and the stack of (spherical) Langlands parameters is invariant under passing to inner forms, we also get a description of \(\mathrm{H}^0\Gamma(\Loc_{\CG}^{\sph}, \Oo_{\Loc_{\CG}^{\sph}})\) for general \(G\).
\begin{proof}
	Since these global sections are given by the \(\widehat{G}^I\)-invariant functions on \(\widehat{G}^Iq^{-1}\sigma\), the proposition follows from the integral Satake isomorphism for \(\Hh_{\Gg}\) \cite[Theorem 10.11]{vdH:RamifiedSatake}, after inverting \(p\).
\end{proof}

In order to relate \(\Loc_{\CG}^{\sph}\) to \(\Loc_{\CG}\), we also need a framed version.
Consider the diagonal action of \(\widehat{G}^I\) on \(\widehat{G}\times \widehat{G}^Iq^{-1}\sigma\), where \(\widehat{G}^I\) acts on \(\widehat{G}\) by right multiplication, and on \(\widehat{G}^{I}q^{-1}\sigma\) by conjugation.

\begin{lem}
	The quotient \((\widehat{G}\times \widehat{G}^Iq^{-1}\sigma)/\widehat{G}^I\) for the above action is representable by a scheme.
\end{lem}
\begin{proof}
	The map \((\widehat{G}\times \widehat{G}^Iq^{-1}\sigma)/\widehat{G}^I \to \widehat{G}/\widehat{G}^I\) is clearly representable by (affine) schemes, so it suffices to show that \(\widehat{G}/\widehat{G}^I\) is a scheme.
	Since \(\widehat{G}^I\) is flat over \(\Spec \IZ[\frac{1}{p}]\) by \cite[Theorem 5.1]{ALRR:Fixed}, this follows from \cite[Théorème 4.C]{Anantharaman:Schemas}.
\end{proof}

\begin{dfn}
	The scheme \(\Loc_{\CG}^{\sph,\square}\) of \emph{framed spherical Langlands parameters} is defined as the quotient \((\widehat{G}\times \widehat{G}^Iq^{-1}\sigma)/\widehat{G}^I\).
	It admits a natural \(\widehat{G}\)-action such that \(\Loc_{\CG}^{\sph,\square}/\widehat{G} \cong \Loc_{\CG}^{\sph}\).
\end{dfn}

Let us now relate \(\Loc_{\CG}^\sph\) to the full stack of Langlands parameters \(\Loc_{\CG}\).
Recall from \cite[§3.1]{Zhu:Coherent} that \(\Loc_{\CG}^{\square}\) admits an open and closed (affine) subscheme
\[\Loc_{\CG,F^t\widetilde{F}/F}^{\square}\colon A\mapsto \left\{\rho\colon \Gamma_{F^t\widetilde{F},\iota}\to \CG(A)\mid d\circ \rho = \chi \iota\colon \Gamma_{F^t\widetilde{F},\iota}\to \IG_m\times \Gal(\widetilde{F}/F) \right\}.\]
Here, we used the following notation:
\begin{itemize}
	\item \(F^t/F\) is the maximal tamely ramified extension,
	\item \(\Gamma_{F^t\widetilde{F},\iota}\) is the pullback of \(\Gal(F^t\widetilde{F}/F)\) along \(\iota\).
	\item \(\chi = (q^{-\parallel\cdot \parallel},\operatorname{pr})\colon W_F\to \IG_m\times \Gal(\widetilde{F}/F)\), where \(\parallel\cdot \parallel\) is the inverse cyclotomic character, and
	\item \(A\) is any \(\IZ[\frac{1}{p}]\)-algebra.
\end{itemize}

By \cite[Lemma 6.3]{Zhu:Ramified} and its proof, there is a closed immersion \(\widehat{G}^Iq^{-1}\sigma \subseteq \Loc_{\CG,F^t\widetilde{F}/F}^{\square}\), given by sending \(g\in \widehat{G}^I(A)\) to the parameter which is trivial on the inertia part of \(\Gamma_{F^t\widetilde{F},\iota}\), and sends the image of \(\sigma\) in \(\Gamma_{F^t\widetilde{F},\iota}\) to \(g\).
The action of \(\widehat{G}\) on \(\Loc_{\CG,F^t\widetilde{F}/F}^{\square}\) then determines a map \(\widehat{G} \times \widehat{G}^Iq^{-1}\sigma\to \Loc_{\CG,F^t\widetilde{F}/F}^{\square}\), which clearly descends to a map
\[\Phi^{\square}\colon \Loc_{\CG}^{\sph,\square}\to \Loc_{\CG,F^t\widetilde{F}/F}^{\square}\subseteq \Loc_{\CG}^{\square}.\]
This map is \(\widehat{G}\)-equivariant, and quotienting out this action gives a canonical map
\[\Phi\colon \Loc_{\CG}^{\sph}\to \Loc_{\CG}.\]
Note that if \(G\) is unramified, this agrees with the usual inclusion \(\Loc_{\CG}^{\mathrm{unr}}\subseteq \Loc_{\CG}\).

\begin{thm}\thlabel{closed immersion}
	The map \(\Phi\colon \Loc_{\CG}^{\sph} \to \Loc_{\CG}\) is a closed immersion.
\end{thm}
\begin{proof}
	It suffices to show that \(\Phi^{\square}\) is a closed immersion.
	Consider the Galois group \(\Gal(F^t\widetilde{F}/F^t)\), i.e., the wild part of \(\Gal(F^t\widetilde{F}/F^t)\).
	This is a finite \(p\)-group, and hence can be obtained as an extension of finite cyclic groups.
	Fix such an extension, as well as lifts \(\{\alpha_i\}_{i=1,\ldots,n}\) of generators of these finite cyclic components.
	Then we have a closed subscheme \(X\subseteq \Loc_{\CG,F^t\widetilde{F}/F}^{\square}\), where
	\[X=\left\{(\beta_1,\ldots,\beta_n,\gamma,\delta)\in (\prod_{i=1}^n \widehat{G}\alpha_i) \times \widehat{G}\tau \times \widehat{G}q^{-1}\sigma\mid \ldots \right\} \subseteq (\prod_{i=1}^n \widehat{G}\alpha_i) \times \widehat{G}\tau \times \widehat{G}q^{-1}\sigma,\]
	where the \(\beta_i,\gamma,\delta\) are subjects to the same relations as the \(\alpha_i,\tau,\sigma\) in \(\Gamma_{F^t\widetilde{F},\iota}\).
	
	By \thref{theorem:thibaud}, the \(\widehat{G}\)-orbit \(Y\subseteq (\prod_{i=1}^n \widehat{G} \alpha_i) \times \widehat{G}\tau\) (for the conjugation action) containing \((\alpha_1,\ldots,\alpha_n,\tau)\in (\prod_{i=1}^n \widehat{G}\alpha_i) \times \widehat{G}\tau\) is a closed subscheme.
	The stabilizer is given by those elements in \(\widehat{G}\) that are fixed by all \(\alpha_i\) as well as \(\tau\), and hence is exactly \(\widehat{G}^I\).
	Let \(Z\) be the preimage of \(Y\) in \(X\) under the projection \(X\to (\prod_{i=1}^n \widehat{G}\alpha_i) \times \widehat{G} \tau\).
	Then there is a natural morphism
	\[\widehat{G} \times \widehat{G}^I \to Z \colon (g,h)\mapsto (g\alpha_1(g^{-1}) \cdot \alpha,\ldots,g\alpha_n(g^{-1})\cdot \alpha_n,g\tau(g^{-1})\cdot \tau,gh\sigma(g^{-1})\cdot \sigma).\]
	This morphism is readily seen to be a \(\widehat{G}^I\)-torsor over \(Z\), for the same action defining \(\Loc_{\CG}^{\sph,\square}\).
	This shows that \(Z\cong \Loc_{\CG}^{\sph,\square}\) is a closed subscheme of \(X\), which itself is a closed subscheme of \(\Loc_{\CG}^{\square}\).
	The inclusion moreover agrees with \(\Phi\), as desired.
\end{proof}

For any representation \(V\) of \(\widehat{G}\) on finite locally free \(\IZ[\frac{1}{p}]\)-modules, we get a vector bundle \(\widetilde{V}\) on \(\Loc_{\CG}\), by pullback along \(\Loc_{\CG}\to */\widehat{G}\).
Similarly, any \(\widehat{G}^I\)-representation on finite locally free modules yields a vector bundle on \(\Loc_{\CG}^{\sph}\) by pullback along \(\Loc_{\CG}^{\sph}\to */\widehat{G}^I\).
Since such vector bundles have applications in the Langlands program (via the geometric Satake equivalence), we explain how these two constructions of vector bundles on \(\Loc_{\CG}^{\sph}\) relate.

\begin{prop}\thlabel{comparing vector bundles}
	For any finite locally free \(\widehat{G}\)-representation \(V\), the vector bundle \(\widetilde{V_{\mid \widehat{G}^I}}\) on \(\Loc_{\CG}^{\sph}\) is canonically isomorphic to the restriction of \(\widetilde{V}\) along \(\Loc_{\CG}^{\sph}\to \Loc_{\CG}\).
\end{prop}
\begin{proof}
	The proposition will follow once we show the following diagram commutes:
	\[\begin{tikzcd}
		\widehat{G}^Iq^{-1}\sigma/\widehat{G}^I \arrow[d] \arrow[r, "\cong"] & {[}(\widehat{G}\times \widehat{G}^Iq^{-1}\sigma)/\widehat{G}^I{]}/\widehat{G} \arrow[d]\\
		*/\widehat{G}^I \arrow[r] & */\widehat{G}.
	\end{tikzcd}\]
	Since the upper arrow is an isomorphism, it suffices to show that \([(\widehat{G}\times \widehat{G}^Iq^{-1}\sigma)/\widehat{G}^I]/\widehat{G} \to */\widehat{G}\) factors through \(*/\widehat{G}^I\).
	This follows by considering the maps
	\[(\widehat{G}\times \widehat{G}^Iq^{-1}\sigma)/\widehat{G}^I \to \widehat{G}/\widehat{G}^I \to \Spec \IZ[\frac{1}{p}],\]
	and quotienting out the natural \(\widehat{G}\)-actions.
\end{proof}

\begin{rmk}
	Our definition of \(\Loc_{\CG}^\sph\) is closely related to a construction from \cite{Haines:Satake,Haines:SatakeCorrection}.
	Namely, after base change to \(\mathbb{C}\), the GIT-quotient of \(\Loc_{\CG,\mathbb{C}}^\sph\) is exactly the parameter space defined in \cite[§5]{Haines:Satake}.
	This also agrees with the target of the map \cite[(2.1)]{Haines:SatakeCorrection}, by the two versions of the Satake equivalence, \cite[Theorem 1.0.1]{HainesRostami:Satake} and \thref{global sections of stack of spherical parameters} (cf.~also \cite[Corollary 5.2]{Haines:Satake}).
	The domain of this map can be viewed as a parameter space for the supercuspidal supports of smooth representations of \(G(F)\) with parahoric fixed vectors.
	Thus, this map can be viewed as a part of the local Langlands correspondence, by mapping a smooth representation with parahoric fixed vectors to its associated Langlands parameter (which turns out to be spherical).
	This construction will also appear prominently in the rest of the article: it is defined via a transfer map on Hecke algebras, which we recall and generalize in \thref{defi transfer}, and we compare it to the local Langlands correspondence constructed in \cite{FarguesScholze:Geometrization} in \thref{Hecke algebras and Bernstein centers}.
	Recall from \cite[Lemma 8.2']{Haines:SatakeCorrection} that this map \cite[(2.1)]{Haines:SatakeCorrection} is an isomorphism for quasi-split groups, but in general it is only finite and birational onto its image.
	For inner forms of \(\GL_n\), this map is even a closed immersion \cite[Lemma 2.6]{Haines:SatakeCorrection}.
\end{rmk}

\begin{rmk}
	The relation between \(\Loc_{\CG}^{\sph}\) and the ramified Satake equivalence \cite{Zhu:Ramified,Richarz:Affine,vdH:RamifiedSatake} also has applications to the categorical local Langlands program, as in \cite{Zhu:Coherent}.
	This will be explained in detail in upcoming work (and in fact, this was the author's original motivation for introducing \(\Loc_{\CG}^{\sph}\)).
\end{rmk}

\section{Integral transfer maps}\label{Sec--Transfer}

In this section, we use the integral Satake and Bernstein isomorphisms from \cite{Boumasmoud:Tale} to define (normalized) transfer homomorphisms between Hecke algebras with integral coefficients, generalizing the case of complex coefficients from \cite[§11.12]{Haines:Stable}, and prove some additional properties.
These will be used to handle groups which are not quasi-split.
Let \(\Lambda\) be a coefficient ring in which \(p\) is invertible, and fix a square root \(\sqrt{q}\in \Lambda\).
Throughout this section, all Hecke algebras are assumed to be \(\Lambda\)-linear.

Let \(G/F\) be a reductive group over a nonarchimedean local field as before, and \(G^*/F\) the (unique up to non-unique isomorphism) quasi-split inner form.
More precisely, we fix a \(\Gal(F^s/F)\)-stable \(G^*_{\adj}(F^s)\)-orbit \(\Psi\) of \(F^s\)-rational isomorphisms \(G\dasharrow G^*\), where the dashed arrow implies it is only defined over the separable closure \(F^s\).
Let \(A\subseteq G\) and \(A^*\subseteq G^*\) be maximal \(F\)-split tori, with centralizers \(M\) and \(T^*\) respectively.
These groups are anisotropic mod center, and we denote by \(\Mm\) and \(\Tt^*\) their unique parahoric \(\Oo\)-models.
We also choose parahorics \(\Pp\) and \(\Pp^*\) of \(G\) and \(G^*\) respectively, for which the facets in the Bruhat--Tits buildings are contained in the apartments corresponding to \(A\) and \(A^*\).
Finally, let \(Q\subseteq G\) be a parabolic subgroup having \(M\) as Levi factor, \(B^*\subseteq G^*\) a Borel containing \(T^*\), and \(B^*\subseteq Q^*\subseteq G^*\) the unique parabolic subgroup corresponding to \(Q\) (cf.~\cite[§11.12.1]{Haines:Stable} for a precise statement).

These data induce Galois-equivariant maps \(Z(\widehat{M}) \cong Z(\widehat{M^*}) \into \widehat{T^*}\).
This in turn induces a morphism 
\[t_{A^*,A}\colon \Lambda[X^*(\widehat{T^*})_I^{\sigma}]^{W(G^*,A^*)} \to \Lambda[X^*(Z(\widehat{M}))_I^\sigma]^{W(G,A)},\] 
where \(W(G,A)\) and \(W(G^*,A^*)\) denote the relative Weyl groups.
Similar to \cite[Lemma 4.77]{Haines:Stable}, we can define a normalized version 
\[\widetilde{t}_{A^*,A}\colon \Lambda[X^*(\widehat{T^*})_I^\sigma]^{W(G^*,A^*)} \to \Lambda[X^*(Z(\widehat{M}))_I^\sigma]^{W(G,A)}\]
\[\sum_{t^*} a_{t^*}t^*\mapsto \sum_m\left(\sum_{t^*\mapsto m} a_{t^*} \delta_{B^*}^{-\frac{1}{2}}(t^*) \delta_Q^{\frac{1}{2}}(m)\right)\cdot m.\]
Here, we implicitly used the identifications \(X^*(\widehat{T^*})_I^\sigma\cong T^*(F)/\Tt^*(\Oo)\) and \(X^*(Z(\widehat{M}))_I^\sigma \cong M(F)/\Mm(\Oo)\) from \cite[Proposition 1.0.2]{HainesRostami:Satake}, and \(\delta_{B^*},\delta_Q\) denote the usual modulus functions (cf.~e.g.~\cite[§2]{Haines:Stable}).

Recall the Satake and Bernstein isomorphisms from \cite[Theorem 6.5.1]{Boumasmoud:Tale}.
Using our fixed root \(\sqrt{q}\in \Lambda\), we use the modulus function to renormalize them and get isomorphisms \(Z(\Hh_{\Pp}) \cong \Lambda[X^*(Z(\widehat{M}))_I^\sigma]^{W(G,A)}\) and \(Z(\Hh_{\Pp^*}) \cong \Lambda[X^*(Z(\widehat{T^*}))_I^\sigma]^{W(G^*,A^*)}\).

\begin{dfn}\thlabel{defi transfer}
	The \emph{normalized transfer homomorphism} \(\widetilde{t}\colon Z(\Hh_{\Pp^*}) \to Z(\Hh_{\Pp})\) is the unique map making the following diagram commute:
	\[\begin{tikzcd}
		Z(\Hh_{\Pp^*}) \arrow[d, "\cong"'] \arrow[r, "\widetilde{t}"] & Z(\Hh_{\Pp}) \arrow[d, "\cong"] \\
		\Lambda[X^*(Z(\widehat{T^*}))_I^\sigma]^{W(G^*,A^*)} \arrow[r, "\widetilde{t}_{A^*,A}"] & \Lambda[X^*(Z(\widehat{M}))_I^\sigma]^{W(G,A)}.
	\end{tikzcd}\]
\end{dfn}

As in \cite[§11.12.2]{Haines:Stable}, this map only depends on \(\Pp^*\) and \(\Pp\), as well as the fixed square root of \(q\).

We now show this transfer map is compatible with parabolic induction.
Let \(\Mm^*\) be the schematic closure of \(M^*\) in \(\Pp^*\), which is a parahoric model of \(M^*\).
Recall also the constant term morphisms
\[c_M^G\colon Z(\Hh_{\Pp})\to Z(\Hh_{\Mm})\colon f\mapsto \left(f^{(Q)}\colon l\mapsto \delta_Q^{\frac{1}{2}}(l) \int_{R} f(lr) \mathrm{d}r\right).\]
Here \(R\) is the unipotent radical of \(Q\), although the constant term morphism is independent of the choice of parabolic \(Q\) \cite[§11.11]{Haines:Stable}.

\begin{lem}\thlabel{transfer parabolic induction}
	Let \(\chi\colon  M(F)\to M(F)/\Mm(\Oo)\to \Lambda^\times\) be a weakly unramified character.
	Then the following diagram commutes:
	\[\begin{tikzcd}
		Z(\Hh_{\Pp^*}) \arrow[d, "c_{M^*}^{G^*}"'] \arrow[r, "\widetilde{t}"] & Z(\Hh_{\Pp}) \arrow[d, "c_M^G"'] \arrow[r] & \End\left(\left(\operatorname{Ind}_{Q(F)}^{G(F)}\left(\chi\otimes \delta_Q^{\frac{1}{2}}\right)\right)^{\Pp(\Oo)}\right)\\
		Z(\Hh_{\Mm*}) \arrow[r, "\widetilde{t}"] & Z(\Hh_{\Mm}) \arrow[r] & \End(\chi). \arrow[u]
	\end{tikzcd}\]
\end{lem}
\begin{proof}
	The compatibility of constant terms and normalized transfer maps can be checked after base change to \(\mathbb{C}\), since all algebras are torsion free, in which case it is shown in \cite[Proposition 4.79]{Haines:Stable}.
	
	On the other hand, the compatibility of constant terms with normalized parabolic induction follows from the discussion in \cite[§11.8]{Haines:Stable}.
\end{proof}

Similarly, we need the transfer maps to be compatible with central isogenies.
Let \(\alpha\colon G'\to G\) be a morphism of reductive \(F\)-groups inducing an isomorphism on adjoint groups, and \(G\dasharrow G^*\) its quasi-split inner form as before, i.e., a \(\Gal(F^s/F)\)-stable \(G_{\mathrm{ad}}^*(F^s)\)-orbit of \(F^s\)-rational isomorphisms.
Fix a representative of this class, and consider the fiber product \((G')^*_{F^s} := G'_{F^s}\times_{G_{F^s}} G^*_{F^s}\).
Then the \(\Gal(F^s/F)\)-action on \(G^*_{F^s}\) induces a \(\Gal(F^s/F)\)-action on \((G')^*_{F^s}\) for which \(\beta_{F^s}\colon (G')^*_{F^s}\to G^*_{F^s}\) is equivariant.
Moreover, we have a \(\Gal(F^s/F)\)-stable \((G'_{\adj})^*(F^s)\)-orbit of isomorphisms \(G'_{F^s}\to (G')^*_{F^s}\).
Thus, \(\beta_{F^s}\) descends to a map \(\beta\colon (G')^*\to G\) which induces an isomorphism on adjoint group, and \((G')^*\) is an inner form of \(G'\) such that the following diagram commutes, where the vertical arrows are suitable choices of representatives:
\[\begin{tikzcd}
	G'_{F^s} \arrow[d, "\cong"'] \arrow[r, "\alpha_{F^s}"] & G_{F^s} \arrow[d, "\cong"]\\
	(G')^*_{F^s} \arrow[r, "\beta_{F^s}"'] & G^*_{F^s}.
\end{tikzcd}\]

Let \(\Pp'\) be the parahoric model of \(G'\) which corresponds to \(\Pp\) under the isomorphism of reduced Bruhat--Tits buildings for \(G\) and \(G'\), and define \((\Pp')^*\) similarly.
Then \(\alpha\) induces a map on Hecke algebras
\[\Hh(\alpha)\colon \Hh_{\Pp'}\to \Hh_{\Pp}\colon f\mapsto \left(\Pp(\Oo)\backslash G(F) / \Pp(\Oo) \ni x\mapsto \sum_{\Pp'(\Oo)\backslash G'(F) / \Pp'(\Oo)\ni y\mapsto x} f(y)\right),\]
and similarly for \(\beta\).

\begin{lem}\thlabel{preservation of centers}
	In the situation above, the map \(\Hh(\alpha)\colon \Hh_{\Pp'}\to \Hh_{\Pp}\) induced by \(\alpha\) is a morphism of \(\Lambda\)-algebras, which preserves centers.
\end{lem}
\begin{proof}
	It is clear that \(\Hh(\alpha)\) is \(\Lambda\)-linear, so let us show it is compatible with the ring structures.
	Let \(\widetilde{W},\widetilde{W}'\) be the Iwahori--Weyl groups of \(G,G'\), and \(W_{\Pp}\cong W_{\Pp'}\) the subgroups corresponding to \(\Pp,\Pp'\); then \(\alpha\) induces a morphism \(\widetilde{W}'\to \widetilde{W}\) preserving these subgroups, which we still denote by \(\alpha\).
	Recall that the convolution of functions \(f,g\in \Hh_{\Pp}\) is given by 
	\[(f\star g)(x) = \int_{G(F)} f(xy^{-1})g(y)\mathrm{d}y = \sum_{w\in W_{\Pp}\backslash \widetilde{W}/W_{\Pp}} \int_{\Pp(\Oo)w\Pp(\Oo)} f(xy^{-1})g(y)\mathrm{d}y,\]
	where \(\mathrm{d}\) is such that \(\Pp(\Oo)\) has measure \(1\).
	By our assumption on \(\alpha\), we have \(\Pp'(\Oo)w'\Pp'(\Oo)/\Pp'(\Oo) \cong \Pp(\Oo)\alpha(w')\Pp(\Oo)/\Pp(\Oo)\), cf.~e.g.~the proof of \cite[Lemma 4.47]{vdH:RamifiedSatake}.
	Thus the difference between \(\Pp'(\Oo)w'\Pp'(\Oo)\) and \(\Pp(\Oo)\alpha(w')\Pp(\Oo)\) is canceled out by the difference between the Haar measures used to define the Hecke algebras, which implies \(\Hh(\alpha)\) is a ring morphism.
	(We note that at Iwahori level, this also follows from the Iwahori--Matsumoto presentation as in \cite[Theorem 3.13.1]{Boumasmoud:Tale}.)
	
	To show that \(\Hh(\alpha)\) preserves centers, we first assume that \(\Pp,\Pp'\) are Iwahori models.
	Then by \cite[Theorem 6.4.1]{Boumasmoud:Tale}, the centers of \(\Hh_{\Pp}\) and \(\Hh_{\Pp'}\) are given by the Weyl group invariants of a subalgebra generated by certain Bernstein elements.
	Since the Weyl groups of \(G\) and \(G'\) agree, it suffices to show \(\Hh(\alpha)\) preserves these Bernstein elements.
	But it follows from the construction in \cite[§4.3]{Boumasmoud:Tale} that the Bernstein element \(\dot{\Theta}_m^\flat\) (in the notation of loc.~cit.) gets sent to \(\dot{\Theta}_{\alpha(m)}^\flat\).
	More precisely, this is similar to \cite[Remark 4.3.2 (i)]{Boumasmoud:Tale}, replacing the use of \cite[Remark 4.1.6]{Boumasmoud:Tale} by the Iwahori--Matsumoto presentation of the Iwahori--Hecke algebra as above.
	
	The result for general parahorics then follows by convolving with \(1_{\Pp(\Oo)}\) and \(1_{\Pp'(\Oo)}\) respectively, as in \cite[Theorem 6.4.1]{Boumasmoud:Tale}, since the quotients of these parahorics by the corresponding Iwahori subgroups yield isomorphic partial flag varieties.
\end{proof}

The desired compatibility can then be stated as follows.

\begin{lem}\thlabel{transfer central isogeny}
	Let \(\alpha\colon G'\to G\) be a morphism of reductive \(F\)-groups inducing an isomorphism on adjoint groups, inducing a morphism \(\beta\colon (G')^*\to G^*\) as above, and \((V,\pi)\) a \(\Lambda\)-linear smooth representation of \(G\) which is generated by its \(\Pp(\Oo)\)-fixed vectors.
	Then the following diagram commutes:
	\[\begin{tikzcd}
		Z(\Hh_{(\Pp')^*}) \arrow[d, "\Hh(\beta)"'] \arrow[r, "\widetilde{t}"] & Z(\Hh_{\Pp'}) \arrow[d, "\Hh(\alpha)"'] \arrow[r] & \End\left(\left(\alpha^*V\right)^{\Pp'(\Oo)}\right)\\
		Z(\Hh_{\Pp^*}) \arrow[r, "\widetilde{t}"'] & Z(\Hh_{\Pp}) \arrow[r] & \End(V^{\Pp(\Oo)}) \arrow[u].
	\end{tikzcd}\]
\end{lem}
\begin{proof}
	First off, the diagram is well-defined by \thref{preservation of centers}.
	
	It follows immediately from the definition of the normalized transfer maps that they are compatible with the maps of Hecke algebras, so that the square on the left commutes.
	
	The commutativity of the rightmost square follows from the equivalence between \(\Lambda\)-linear smooth \(G(F)\)-representations \((V,\pi)\) generated by their \(\Pp(\Oo)\)-fixed vectors, and modules under the Hecke algebra \(\Hh_{\Pp}\), which is given by
	\[(V,\pi)\mapsto \left(V^{\Pp(\Oo)},f\mapsto\left(v\mapsto \sum_{g\in G(F)/\Pp(\Oo)} f(g)\pi(g)v\right)\right).\]
\end{proof}

\section{On the Fargues--Scholze correspondence for representations\\ with parahoric fixed vectors}\label{Sec--FS}

In \cite[§IX.4]{FarguesScholze:Geometrization}, Fargues--Scholze attach to each irreducible smooth representation \(\pi\) of \(G(F)\) with \(\overline{\IQ}_\ell\)- (or even \(\overline{\IF}_\ell\)-)coefficients a semisimple Langlands parameter \(\phi_\pi^{\mathrm{FS}}\).
In the current section, we make explicit this construction for representations generated by their parahoric-fixed vectors.
Some of the results below have previously been obtained by Li \cite{Li:Compatibility}, cf.~\thref{remark conjecture Haines}.

Let \(G/F\) be as before, and let \(\Lambda\) be a \(\IZ_\ell[\sqrt{q}]\)-algebra, for some prime \(\ell\neq p\).
We moreover assume that the order of \(\pi_0Z(G)\) is invertible in \(\Lambda\).
Again, all Hecke algebras will be assumed to have coefficients in \(\Lambda\).
In this section and the next, we will use the fixed square root of \(q\) to implicitly identify the stack \(\Loc_{\CG,\Lambda}\), defined using the C-group, with the stacks of Langlands parameters constructed in \cite{DatHelmKurinczukMoss:Moduli,FarguesScholze:Geometrization} (although we will keep the notation \(\Loc_{\CG}\)).
Recall the following notions of Bernstein centers.

\begin{dfn}{\cite[Definition IX.0.2]{FarguesScholze:Geometrization}}
	\begin{enumerate}
		\item The \emph{Bernstein center} of \(G(F)\) is
		\[\Zz(G(F),\Lambda) = \pi_0\End(\identity_{\Dd(G(F),\Lambda)}),\]
		where \(\Dd(G(F),\Lambda)\) is the derived \(\infty\)-category of smooth \(\Lambda\)-linear representations of \(G(F)\).
		\item The \emph{spectral Bernstein center} of \(G\) is \[\Zz^{\mathrm{spec}}(G,\Lambda) = \mathrm{H}^0\Gamma(\Loc_{\CG}, \Oo_{\Loc_{\CG}}) =  \Oo(\Loc_{\CG,\Lambda}^{\square})^{\widehat{G}}.\]
	\end{enumerate}
\end{dfn}

Then under our assumption on \(\Lambda\), Fargues--Scholze \cite[Corollary IX.0.3]{FarguesScholze:Geometrization} have constructed a map
\[\Psi_G\colon \Zz^{\mathrm{spec}}(G,\Lambda)\to \Zz(G(F),\Lambda);\]
in case \(\Lambda=\overline{\IQ}_\ell,\overline{\IF}_\ell\), this induces a correspondence from irreducible smooth \(G(F)\)-representations with to semisimple Langlands parameters.
For tori, this correspondence is known to agree with local class field theory, and is hence well-understood.
Since we will also consider groups which are not quasi-split, whose minimal Levi's are not tori, we start with the following lemma, studying the correspondence for groups which are anisotropic modulo their center.

\begin{lem}\thlabel{FS for anisotropic groups}
	Let \(M/F\) be a reductive group which is anisotropic mod center, \(\Mm/\Oo\) its unique parahoric subgroup, and \(\chi\colon M(F)\to M(F)/\Mm(\Oo)\to \overline{\IQ}_\ell^\times\) a weakly unramified character.
	Then the Fargues--Scholze parameter \(\phi_\chi^{\mathrm{FS}}\) associated with \(\chi\) is spherical.
\end{lem}
\begin{proof}
	We need to show that \(\phi_\chi^{\mathrm{FS}}\) has a representative which is trivial on inertia.
	Let \(\widetilde{M}\to M\) be the canonical central extension with simply connected derived subgroup whose kernel is an induced torus, as in \cite[§5.3.1]{FHLR:Singularities} (note that the construction in loc.~cit.~works for general fields).
	Since this induces a closed immersion on dual groups by \thref{closed immersion on dual groups}, the compatibility of the Fargues--Scholze construction with central isogenies, \cite[Theorem IX.6.1]{FarguesScholze:Geometrization}, reduces us to consider \(\widetilde{M}\), so that we may assume \(M\) has simply connected derived subgroup.
	
	Let \(D=M/M_{\der}\) be the maximal torus quotient of \(M\).
	Then we have an exact sequence 
	\[1\to M_{\der} \to M \to D\to 1,\]
	and hence an exact sequence
	\[1\to \pi_1(M_{\der})\to \pi_1(M)\to \pi_1(D)\to 1.\]
	Since \(M_{\der}\) is simply connected, taking Frobenius-invariants gives an isomorphism \(\pi_1(M)^\sigma\cong \pi_1(D)^\sigma\).
	Thus, the weakly unramified characters of \(M(F)\) and \(D(F)\) are in bijection by \cite[Proposition 1.0.2]{HainesRostami:Satake}.
	The compatibility of the Fargues--Scholze correspondence with twisting and local class field theory in the case of tori, \cite[Theorem I.9.6 (i)-(ii)]{FarguesScholze:Geometrization}, then reduces us to showing the proposition for a single weakly unramified character of \(M\).
	By the compatibility with central isogenies, we may thus assume \(M\) is adjoint.
	By the compatibility with products and Weil restrictions of scalars \cite[Propositions IX.6.2 and IX.6.3]{FarguesScholze:Geometrization}, we may further assume \(M\) is absolutely simple.
	
	Now, any absolutely simple anisotropic adjoint group is of the form \(\PGL_{1,D}\) for some division algebra \(D\) with center \(F\) \cite[4.6]{BruhatTits:Groupes3}.
	Then considering the central isogeny \(\GL_{1,D}\to \PGL_{1,D}\), which induces an injection on dual groups, \cite[Theorem IX.6.1]{FarguesScholze:Geometrization} allows us to consider the case \(M=\GL_{1,D}\).
	Then it is known that the semisimple Langlands parameter associated with a smooth irreducible representation of \(D^\times\) agrees with the parameter associated with its local Jacquet--Langlands transfer as a representation of \(\GL_d(F)\), where \(\dim_F D = d^2\): see \cite[Theorem 6.6.1]{HansenKalethaWeinstein:Kottwitz} for the mixed characteristic case, and \cite[Theorem D]{LiHuerta:LocalGlobal} for the equicharacteristic case.
	By the above paragraph, it suffices to consider a single weakly unramified character of \(D^\times\), so we may choose the trivial one, whose local Jacquet--Langlands transfer is the Steinberg representation of \(\GL_d(F)\).
	The compatibility with parabolic induction then shows that the associated Langlands parameter is spherical.
\end{proof}

We have used the following lemma.

\begin{lem}\thlabel{closed immersion on dual groups}
	Let \(\widetilde{G}\to G\) be a central extension of reductive \(F\)-groups, whose kernel is a torus.
	Then this induces a canonical map of dual groups, which is a closed immersion.
\end{lem}
\begin{proof}
	Any morphism of reductive groups which induces an isomorphism on adjoint groups, yields a (contravariant) morphism of dual groups.
	So it suffices to show this map is a closed immersion.
	Consider the affine Grassmannian \(\Gr_G\) of \(G\) from \cite[§III.3]{FarguesScholze:Geometrization}.
	Then the induced map \(\Gr_{\widetilde{G}}\to \Gr_G\) is surjective by \cite[Lemma III.3.5]{FarguesScholze:Geometrization}.
	More precisely, this map induces an isomorphism of each connected component of \(\Gr_{\widetilde{G}}\) onto its image, and the induced map on connected components is surjective.
	The lemma then follows from the Satake equivalence \cite[§VI]{FarguesScholze:Geometrization} and the Tannakian formalism, since pushforward along \(\Gr_{\widetilde{G}}\to \Gr_G\) preserves equivariant perverse sheaves and is compatible with convolution.
\end{proof}

Now, let \(\Ii/\Oo\) be an Iwahori-model of \(G\) corresponding to an alcove \(\aa_0\) in the Bruhat--Tits building of \(G\), and \(\Gg/\Oo\) a special parahoric model such that the corresponding special vertex lies in the closure of \(\aa_0\).
We also fix a quasi-split inner form \(G^*\) of \(G\), together with a very special parahoric integral model \(\Gg^*\), and corresponding (\(\Lambda\)-linear) Hecke algebra \(\Hh_{\Gg^*}\).
Since the stack of Langlands parameters is invariant under passing to inner forms, there is an obvious map \(\Zz^{\mathrm{spec}}(G,\Lambda)\to \Hh_{\Gg^*}\), obtained by Propositions \ref{global sections of stack of spherical parameters} and \ref{closed immersion}.
The interpretation of the Iwahori--Hecke algebra \(\Hh_{\Ii}\) as \(\End_{G(F)}(\cInd_{\Ii(\Oo)}^{G(F)} \Lambda)\) also yields a map \(\Zz(G(F),\Lambda)\to \Hh_{\Ii}\).
Finally, recall the Bernstein map
\[\Theta_{\mathrm{Bern}}\colon \Hh_{\Gg}\to \Hh_{\Ii}\]
from  \cite[Theorem 6.4.1]{Boumasmoud:Tale}, which is injective with image the center \(Z(\Hh_{\Ii})\) of \(\Hh_{\Ii}\).
The following proposition makes explicit the Fargues--Scholze correspondence for irreducible representations with Iwahori (and hence parahoric) fixed vectors.

\begin{prop}\thlabel{Hecke algebras and Bernstein centers}
	The following diagram commutes:
	\[\begin{tikzcd}
		\Zz^{\mathrm{spec}}(G,\Lambda) \arrow[r, "\Psi_G"] \arrow[d] & \Zz(G(F),\Lambda) \arrow[dd] \\
		\Hh_{\Gg^*} \arrow[d, "\widetilde{t}"']&\\
		\Hh_{\Gg} \arrow[r, "\Theta_{\mathrm{Bern}}"'] & \Hh_{\Ii}.
	\end{tikzcd}\]
\end{prop}
\begin{proof}
	First, note that the rightmost map takes values in the center of \(\Hh_{\Ii}\), which we identify with \(\Hh_{\Gg}\) via \(\Theta_{\mathrm{Bern}}\).
	By the integral Satake isomorphism as in \cite[Theorem 5.2.1]{Boumasmoud:Tale}, \(\Hh_{\Gg}\) is a subring of an integral domain, and hence itself a domain.
	Moreover, by \cite[Theorem VIII.1.3]{FarguesScholze:Geometrization}, \(\Zz^{\mathrm{spec}}(G,\Lambda)\) is torsionfree. 
	Hence, to show this diagram commutes we may work with \(\overline{\IQ}_\ell\)-coefficients.
	
	Note that \(\Zz^{\mathrm{spec}}(G,\overline{\IQ}_\ell) \to \Hh_{\Gg^*}\otimes_{\Lambda} \overline{\IQ}_\ell\) is surjective.
	Indeed, this follows since \(\Loc_{\CG,\overline{\IQ}_\ell}^{\sph}\to \Loc_{\CG,\overline{\IQ}_\ell}\) is a closed immersion by \thref{closed immersion}, and taking \(\widehat{G}\)-invariants is exact in characteristic 0.
	We claim that the composite \[\Zz^{\mathrm{spec}}(G,\overline{\IQ}_\ell) \to \Zz(G(F),\overline{\IQ}_\ell) \to \Hh_{\Gg}\otimes_{\Lambda} \overline{\IQ}_\ell\] factors through \(\Zz^{\mathrm{spec}}(G,\overline{\IQ}_\ell) \to \Hh_{\Gg^*}\otimes_{\Lambda} \overline{\IQ}_\ell\).
	Since these rings are reduced, the statement for the corresponding affine schemes can be checked on \(\overline{\IQ}_\ell\)-valued points, so we need to show that the Fargues--Scholze parameter attached to a smooth irreducible \(G(F)\)-representation with parahoric fixed vectors is spherical.
	Let \(M\) be the centralizer of a maximal split torus of \(G\), which is a minimal Levi.
	Then any \(G(F)\)-representation which is generated by its Iwahori-fixed vectors is a subquotient of the (normalized) parabolic induction of an unramified character of \(M\) \cite[Proposition 2.6]{Casselman:UnramifiedI}. 
	Thus, the compatibility of \(\Psi_G\) with parabolic induction \cite[Corollary IX.7.3]{FarguesScholze:Geometrization} reduces us to the case \(G=M\), where \(M\) is anisotropic mod center.
	(Although \cite[Corollary IX.7.3]{FarguesScholze:Geometrization} is concerned with unnormalized parabolic inductions, it also considers the unnormalized constant term morphisms. 
	Normalizing both still yields a commutative diagram as in loc.~cit.)
	Then the claim follows from \thref{FS for anisotropic groups}.
	
	It remains to show that this factorization is given by the transfer map \(\widetilde{t}\otimes_{\Lambda} \overline{\IQ}_\ell\).
	Again, this can be done on \(\overline{\IQ}_\ell\)-points, and we need to show the two ways of attaching Langlands parameters to irreducible representations with parahoric fixed vectors, as in \cite{FarguesScholze:Geometrization} and \cite{Haines:Satake} agree.
	Note that both constructions are compatible with central isogenies, parabolic induction, products, Weil restrictions of scalars, and twisting.
	Indeed, for the Fargues--Scholze construction, this is \cite[Theorem I.9.6]{FarguesScholze:Geometrization}.
	For the construction involving the transfer maps, the compatibility with parabolic induction and central isogenies is shown in Lemmas \ref{transfer parabolic induction} and \ref{transfer central isogeny}, whereas the compatibility with products is straightforward.
	Note that Hecke algebras and the spectral Bernstein center are invariant under Weil restriction of scalars (cf.~\cite[Lemma 3.1.15]{Zhu:Coherent} for the latter), and it is straightforward to check that the left vertical arrows in the diagram above are preserved under these canonical identifications (this is done in detail in \cite[Lemma 5.1]{Li:Compatibility}).
	This shows the compatibility with Weil restriction of scalars.
	Along with the compatibility with class field theory in the case of tori (\cite[§13.1]{Haines:Satake}), one obtains the compatibility with twisting from a classical argument (e.g.~\cite[End of §IX.6]{FarguesScholze:Geometrization}).
	As in the proof of \thref{FS for anisotropic groups}, one then reduces to showing the proposition for the trivial representation of \(D^\times\), for some division algebra \(D\) with center \(F\), or equivalently, its Jacquet--Langlands transfer to a representation of \(\GL_d(F)\).
	In that case, both correspondences agree with the classical Langlands correspondence: \cite[Theorem IX.7.4]{FarguesScholze:Geometrization} and \cite[§13.2]{Haines:Satake}.
\end{proof}

After attaching Satake parameters to irreducible representations of \(G(F)\) with parahoric fixed vectors \cite{Haines:Satake}, Haines conjectured that his construction is compatible with the local Langlands correspondence.
The proposition above shows this compatibility for the local Langlands correspondence constructed in \cite{FarguesScholze:Geometrization}.

\begin{cor}
	The Fargues--Scholze local Langlands correspondence satisfies \cite[Conjecture 13.1]{Haines:Satake}.
\end{cor}

\begin{rmk}\thlabel{remark conjecture Haines}
	After the first draft of this paper was completed, Tom Haines pointed out to us that the above corollary was already known in the case of \(p\)-adic fields: this is shown in \cite{Li:Compatibility}, via a strategy similar to ours, although with different execution.
	Nevertheless, we have given the proof above since it fits the discussion, it works for arbitrary nonarchimedean local fields (although the arguments of \cite{Li:Compatibility} combined with \cite[Theorem D]{LiHuerta:LocalGlobal} would also show the corollary for local function fields), and we will need the precise formulation of \thref{Hecke algebras and Bernstein centers}, including the case of integral coefficients.
\end{rmk}

\section{Congruence relations}\label{Sec--Congruence}

With the above preparations at hand, we are ready to prove the Eichler--Shimura congruence relations for Hodge type Shimura varieties.
This will be done by applying the spectral action from \cite{FarguesScholze:Geometrization}, based on the approach in \cite{DvHKZ:Igusa}.
The key for this is the construction of Igusa stacks, although they will not explicitly appear below.

We fix the following setup.
Let \((\mathsf{G},\mathsf{X})\) be a Hodge type Shimura datum with Hodge cocharacter \(\mu\) and reflex field \(\mathsf{E}\).
Let \(v\) be a place of \(\mathsf{E}\) lying over \(p\), and let \(E:=\mathsf{E}_v\) and \(G=\mathsf{G}_{\IQ_p}\).
Let \(q\) denote the residue cardinality of \(E\).
Fix a special parahoric subgroup \(K_p\subseteq \mathsf{G}(\IQ_p)\), as well as an Iwahori \(K_p'\subseteq K_p\).
For any neat compact open subgroup \(K^p\subseteq \mathsf{G}(\IA_f^p)\), we let \(K:=K_pK^p\subseteq \mathsf{G}(\IA_f)\) and \(K':=K_p'K^p\).
Then the Shimura variety associated to \((\mathsf{G},\mathsf{X})\) at level \(K\) (resp.~\(K'\)) has a canonical model over \(\mathsf{E}\), and we denote its base change to \(E\) by \(\Sh_K\) (resp.~\(\Sh_{K'}\)).

Fix a prime \(\ell\neq p\) which is prime to \(|\pi_0(Z(G))|\), let \(\Lambda\) be a \(\IZ_\ell\)-algebra with a fixed square root of \(p\), and base change all stacks of Langlands parameters along \(\IZ[\frac{1}{p}]\to \Lambda\).
Denote by \(\Gg\) (resp.~\(\Ii\)) the parahoric integral model of \(G\) with \(\Gg(\IZ_p) = K_p\) (resp.~\(\Ii(\IZ_p) = K_p'\)), and by \(\Hh_{\Gg} = C_c(\Gg(\IZ_p)\backslash G(\IQ_p)/\Gg(\IZ_p),\Lambda)\) (resp.~\(\Hh_{\Ii} = C_c(\Ii(\IZ_p)\backslash G(\IQ_p)/\Ii(\IZ_p),\Lambda)\)) the corresponding Hecke algebra with \(\Lambda\)-coefficients.

\begin{rmk}
	The Hecke algebras \(\Hh_{\Gg}\) and \(\Hh_{\Ii}\) naturally act on \(R\Gamma(\Sh_{K,\overline{E}}, \Lambda)\) and \(R\Gamma(\Sh_{K',\overline{E}},\Lambda)\) respectively.
	We will always consider the \emph{right} actions of these Hecke algebras, so that the Hecke polynomial for \(\mu\) will show up, rather than \(\mu^{-1}\) \cite[Remark 1.2]{Koshikawa:Eichler-Shimura}.
\end{rmk}

In order to state and prove the Eichler--Shimura relations, we need to introduce two Hecke polynomials.
Let \(E^{\mathrm{unr}}/E\) be the maximal unramified extension, and \(\sigma_E\in \Gal(E^{\mathrm{unr}}/E)\) the arithmetic Frobenius.

\begin{dfn}
	Consider the group \(\widehat{G}^I\rtimes \Gal(E^{\mathrm{unr}}/E)\).
	It admits a unique \(\Lambda\)-linear locally free representation \(r\) such that the restriction of \(r\) to \(\widehat{G}^I\) agrees with the restriction of the representation of \(\widehat{G}\) with highest weight \(\mu\), and \(\Gal(E^{\mathrm{unr}}/E)\) acts trivially on the highest weight space.
	Let \(d\) be the rank of this representation (which agrees with the dimension of \(\Sh_K\)).
	
	In case \(G=G^*\) is quasi-split, the associated (renormalized) \emph{Hecke polynomial} is then defined as
	\[H_\mu^*:= \det(X-q^{\frac{d}{2}}r(g,\sigma_E))\in \Hh_{\Gg^*}[X],\]
	using \thref{global sections of stack of spherical parameters}.
	In general, it is defined as the image \(H_\mu:=\widetilde{t}(H_\mu^*)\in \Hh_{\Gg}[X]\) of \(H_\mu^*\) under the normalized transfer map.
\end{dfn}

\begin{rmk}
	\begin{enumerate}
		\item As we are working with integral coefficients, we have to be careful about highest weight representations.
		But since \(\mu\) is minuscule (as a character of \(\widehat{G}\)), this is not an issue, and we just mean the representation which has a unique rank 1 weight space for each element in the Weyl group orbit of \(\mu\).
		\item In case \(G\) is quasi-split, one can use the Satake equivalence to give a geometric construction of this representation. 
		For example, under the version from \cite[Corollary 9.6]{vdH:RamifiedSatake} (and after applying the étale realization functor for motives, which corresponds to forgetting the \(\IG_m\)-action in the formulation of loc.~cit.), it is the representation corresponding to the nearby cycles (in the sense of \cite[§6]{AGLR:Local}) of the shifted constant sheaf on the Schubert cell corresponding to \(\mu\), in the affine Grassmannian from \cite{FarguesScholze:Geometrization}.
		\item We only obtain an action of the unramified part of the Galois group of \(E\) on \(\widehat{G}^I\), not on \(\widehat{G}\).
		But this has the advantage that we do not need to choose a lift of Frobenius.
		\item For unramified groups, this clearly recovers the Hecke polynomial from e.g.~\cite[§6]{BlasiusRogawski:Zeta}.
	\end{enumerate}
\end{rmk}

Recall also the spectral Hecke polynomial from \cite[Definition 9.3.2]{DvHKZ:Igusa}.
For a locally free \(\Lambda\)-linear representation \(r_V\colon \widehat{G}\rtimes W_E\to \GL(V)\) and \(\gamma\in W_E\), it is defined as
\[H_{G,V,\gamma}^{\mathrm{spec}} := \det(X-r_V\circ \phi_E^{\mathrm{univ}}(\gamma)) = \sum_{i=0}^{\operatorname{rk} V} (-1)^i \chi_{\wedge^i V} X^{\operatorname{rk}V-i}\in \Zz^{\mathrm{spec}}(G,\Lambda)[X].\]
Here, \(\chi_V\in \Zz^{\mathrm{spec}}(G,\Lambda)\) is the function sending an L-parameter \(\phi\) to the trace of \(r_V\circ \phi(\gamma)\), and \(\phi_E^{\mathrm{univ}}\) is the restriction of the universal L-parameter to \(W_E\subseteq W_{\IQ_p}\).

We are now ready to prove the main theorem of this section, concerning the action of the Weil group \(W_E\) of \(E\) on the cohomology \(\mathrm{H}^i(\Sh_{K,\overline{E}}, \Lambda)\).

\begin{thm}\thlabel{Eichler-Shimura}
	For each \(i\in \IZ\), the inertia subgroup \(I_E\subseteq W_E\) acts unipotently on \(\mathrm{H}_{\mathrm{\acute{e}t}}^i(\Sh_{K',\overline{E}},\Lambda)\).
	Moreover, for any Frobenius lift \(\sigma_E\in W_E\), the action of \(\sigma_E\) on \(R\Gamma(\Sh_{K',\overline{E}},\Lambda)\) satisfies \(H_\mu(\sigma_E) = 0\).
\end{thm}
\begin{proof}
	Let \(\widetilde{V_\mu}\) be the vector bundle on \(\Loc_{\CG,\Lambda}\) obtained by the representation \(r_\mu\) of \(\widehat{G}\rtimes W_E\) defined in \cite[§8.4.8]{DvHKZ:Igusa}; in particular its restriction to \(\widehat{G}\) has highest weight \(\mu\).
	By \cite[Corollary 9.1.5]{DvHKZ:Igusa}, \(\pi_0\End_{\Loc_{\CG,\Lambda}}(\widetilde{V_\mu})\) acts on \(R\Gamma(\Sh_{K,\overline{E}},\Lambda)\), and the restricted action of \(W_E\subseteq \pi_0\End_{\Loc_{\CG,\Lambda}}(\widetilde{V_\mu})\) agrees with the usual \(W_E\)-action on the cohomology of \(\Sh_{K,\overline{E}}\).
	Let \(\gamma\in I_E\) be an element in the inertia subgroup.
	By \thref{Hecke algebras and Bernstein centers}, smooth \(G(\IQ_p)\)-representations generated by their Iwahori-fixed vectors yield spherical Langlands parameters.
	Since, up to \(\widehat{G}\)-conjugation, these are trivial on \(I_E\), the spectral Hecke polynomial attached to \(V_\mu\) and \(\gamma\) is \(\det(X-\identity)\).
	By \cite[Corollary 9.3.3]{DvHKZ:Igusa} (and since each \(\mathrm{H}_{\mathrm{\acute{e}t}}^i(\Sh_{K',\overline{E}},\Lambda)\) is generated by its \(K_p'\)-fixed vectors), we see that \((\gamma-\identity)^{\operatorname{rk} V_\mu}\) acts trivially on \(\mathrm{H}_{\mathrm{\acute{e}t}}^i(\Sh_{K',\overline{E}},\Lambda)\).
	In other words, \(\gamma\) acts unipotently.
	
	Next, let \(\sigma\in W_E\) be a lift of Frobenius.
	Then by \thref{Hecke algebras and Bernstein centers}, the image of \(H_{G,V_\mu,\sigma_E}^{\mathrm{spec}}\in \Zz^{\mathrm{spec}}(G,\Lambda)[X]\) in \(\Hh_{\Gg}[X]\) agrees with \(\widetilde{t}(\det(X-r_\mu\circ \phi_E^{\mathrm{univ}}(\sigma_E)))\).
	By \cite[Corollary 9.3.5]{DvHKZ:Igusa}, the evaluation of this polynomial at \(\sigma_E\) acts trivially on \(R\Gamma(\Sh_{K',\overline{E}},\Lambda)(\frac{d}{2})\).
	This implies that the renormalized Hecke polynomial \(H_\mu(\sigma_E)\) acts trivially on the non-twisted complex \(R\Gamma(\Sh_{K',\overline{E}},\Lambda)\), where we use \thref{comparing vector bundles} to relate \(H_\mu^*\) (and hence \(H_\mu\)) with \(\det(X-r_\mu\circ \phi_E^{\mathrm{univ}}(\sigma_E))\).
\end{proof}
 

\appendix

\section{The orbit of a pinning-preserving automorphism, by Sean Cotner}\label{Appendix:Sean}

This appendix aims to prove \thref{theorem:thibaud}, which is used to prove \thref{closed immersion} in the main text. If $S$ is a scheme and $\Gamma$ is a finite group acting on an $S$-group scheme $G$, then we let $Z^1(\Gamma, G)$ denote the scheme of $1$-cocycles $\Gamma \to G$, i.e., homomorphisms $\Gamma \to G \rtimes \Gamma$ whose composition with the projection $G \rtimes \Gamma \to \Gamma$ is the identity map on $\Gamma$.
Throughout this appendix, the notation will deviate from the main text, and \(G\) will denote a split reductive group over a scheme \(S\).

\begin{thm}\thlabel{theorem:thibaud}
	Let $S$ be a scheme, let $G$ be a split reductive $S$-group scheme, and let $\Gamma$ be a finite group of $S$-automorphisms preserving a common pinning of $G$. The map $i\colon G/G^\Gamma \to Z^1(\Gamma, G)$, $g \mapsto (\gamma \mapsto g\gamma(g)^{-1})$, is a closed embedding.
\end{thm}

Note that by \cite[Theorem 1.1(1)]{ALRR:Fixed}, the fixed point scheme $G^\Gamma$ is a finitely presented flat closed $S$-subgroup scheme of $G$, and thus the quotient $G/G^\Gamma$ is a finitely presented separated algebraic space by \cite[Corollary 6.3]{Artin:Versal}. The map $i$ is clearly monic, so by \cite[Proposition 8.11.5]{EGA4.3} it suffices to verify that $i$ satisfies the valuative criterion of properness. As we will see in the proof, the essential claim is Lemma~\ref{lemma:generic-pinning-gives-pinning}.

The proof of the following lemma was shown to me by Loren Spice; it will be significantly generalized to a statement about groups of pinning-preserving automorphisms in upcoming work of Adler--Lansky--Spice \cite{ALS-3}. I thank the authors for the permission to include their proof.

\begin{lem}\label{lemma:als-pinning}
	Let $k$ be an algebraically closed field, let $G$ be a connected reductive group over $k$, and let $\Gamma$ be a finite group of $k$-automorphisms which stabilizes some pinning of $G$. If $(B, T)$ is a $\Gamma$-stable Borel pair of $G$, then it extends to a $\Gamma$-stable pinning $(B, T, \{X_\alpha\})$.
\end{lem}

\begin{proof}
	Let $(B_0, T_0, \{X_\alpha\})$ be a $\Gamma$-stable pinning; it is clearly enough to show that $(B, T)$ is $G^\Gamma(k)$-conjugate to $(B_0, T_0)$. By \cite[Proposition 3.5(iii)]{AdlerLansky:Lifting}, the groups $(T \cap G^\Gamma)^0_{\redu}$ and $(T_0 \cap G^\Gamma)^0_{\redu}$ are both maximal tori of $(G^\Gamma)^0_{\redu}$ and we have $T = Z_G((T \cap G^\Gamma)^0_{\redu})$ and $T_0 = Z_G((T_0 \cap G^\Gamma)^0_{\redu})$, so by $G(k)$-conjugacy of maximal tori in $(G^\Gamma)^0_{\redu}$ we see that $T$ is $G^\Gamma(k)$-conjugate to $T_0$. We may therefore assume $T = T_0$.
	
	Next, recall that the map $w \mapsto wBw^{-1}$ defines a bijection from the Weyl group $W(G, T)$ to the set of Borel subgroups of $G$ containing $T$. This bijection is $\Gamma$-equivariant, so it defines a bijection from $W(G, T)^\Gamma$ to the set of $\Gamma$-stable Borel subgroups of $G$ containing $T$. By \cite[Corollary 15]{AdlerLansky:Root}, the natural map $W((G^\Gamma)^0_{\redu}, T \cap (G^\Gamma)^0_{\redu}) \to W(G, T)^\Gamma$ is bijective, so it follows that there is some $g \in N_{G^\Gamma}(T \cap (G^\Gamma)^0_{\redu})(k)$ such that $gBg^{-1} = B_0$, as desired.
\end{proof}

\begin{lem}\label{lemma:chevalley-steinberg-replacement}
	Let $k$ be an algebraically closed field, let $G$ be a connected reductive group over $k$, and let $\Gamma$ be a finite group of automorphisms preserving a Borel pair $(B, T)$ of $G$. The map $Z^1(\Gamma, T)/\!/N_G(T) \to Z^1(\Gamma, G)/\!/G$ is radicial, i.e., it is injective on $k$-points.
\end{lem}

\begin{proof}
	Let $\alpha, \beta\colon \Gamma \to T$ be two $1$-cocycles which are $G(k)$-conjugate. Note that $\alpha$ and $\beta$ induce actions of $\Gamma$ on $G$ which preserve the pair $(B, T)$. If $g\beta g^{-1} = \alpha$, then in particular $gT^\Gamma g^{-1} \subset G^{\alpha(\Gamma)}$: indeed, $\beta(\Gamma)$ centralizes $T^\Gamma$, so $\alpha(\Gamma) = g\beta(\Gamma)g^{-1}$ centralizes $gT^\Gamma g^{-1}$. Thus by conjugacy of maximal tori in the algebraic $k$-group $G^{\alpha(\Gamma)}$, there is some $h \in G^{\alpha(\Gamma)}(k)$ such that $hg$ normalizes $T^\Gamma$ and $hg\beta g^{-1} h^{-1} = \alpha$. Thus we may replace $g$ by $hg$ to assume $g$ normalizes $T^\Gamma$. By \cite[Proposition 3.5(iii)]{AdlerLansky:Lifting}, we have $T = Z_G(T^\Gamma)$, so in fact $g$ normalizes $T$ and we are done.
\end{proof}

\begin{lem}\label{lemma:generic-pinning-gives-pinning}
	Let $A$ be a discrete valuation ring with fraction field $K$. If $G$ is a split reductive $A$-group scheme and $\Gamma$ is an $A$-automorphism of $G$ which preserves a pinning of $G_K$, then, after possibly extending $A$, there is a pinning of $G$ preserved by $\Gamma$.
\end{lem}

\begin{proof}
	Let $(B_0, T_0, \{X_\alpha\})$ be a pinning of $G_K$ preserved by $\Gamma$. By properness of the scheme of Borels \cite[Theorem 5.2.11(3)]{Conrad:Reductive}, there is a Borel $A$-subgroup $B$ of $G$ with generic fiber $B_0$, and this $B$ is evidently $\Gamma$-stable. Let $(G^\Gamma)'$ denote the schematic closure of $((G_K)^\Gamma)^0_{\redu}$ in $G$; by extending $A$, we may and do assume that $((G_K)^\Gamma)^0_{\redu}$ is a smooth $K$-group scheme, so $(G^\Gamma)'$ is a closed $A$-subgroup scheme of $G$ with connected reductive generic fiber. By \cite[Theorem 1.1(1) and (4)]{ALRR:Fixed}, the $A$-group scheme $(G^\Gamma)'$ is quasi-reductive in the sense of \cite{PrasadYu:Quasi}. By \cite[Proposition 3.4]{PrasadYu:Quasi}, after possibly extending $A$ we may assume that the normalization $\widetilde{(G^\Gamma)'}$ is a reductive $A$-group scheme. Let $\pi\colon \widetilde{(G^\Gamma)'} \to (G^\Gamma)'$ be the normalization morphism, so $\pi_K$ is an isomorphism and $\pi_k$ is a unipotent isogeny by \cite[Proposition 4.3]{PrasadYu:Quasi}. By \cite[Proposition 3.5(ii)]{AdlerLansky:Lifting}, the pullback $\pi^{-1}(B_0)$ is a Borel $K$-subgroup of $\widetilde{(G^\Gamma)'}_K$. Since $\widetilde{(G^\Gamma)'}$ is reductive, another application of properness of the scheme of Borels shows that there is a unique Borel $A$-subgroup $B_1$ of $\widetilde{(G^\Gamma)'}$ with generic fiber $\pi^{-1}(B_0)$. Let $T_1 \subset B_1$ be a maximal $A$-torus, and note that the restriction map $\pi|_{T_1} \colon T_1 \to G$ is monic by applying the fibral isomorphism criterion to the identity section $\Spec A \to \ker \pi|_{T_0}$ (using that $\pi_k$ has unipotent kernel). By \cite[Proposition 3.5(iii)]{AdlerLansky:Lifting}, the centralizer $T \coloneqq Z_G(\pi(T_1))$ has fibers which are maximal subtori of the fibers of $G$, and it follows that $T$ is a $\Gamma$-stable maximal $A$-torus in $G$.
	
	By Lemma~\ref{lemma:als-pinning} and the fact that $\Gamma$ preserves \textit{some} pinning of $G_K$, it follows that $\Gamma$ preserves $(B, T, \{X_\alpha\}_{\alpha \in \Delta})$ for some choice of basis elements $X_\alpha \in (\Lie U_\alpha)(K)$. We claim that it is possible to choose each $X_\alpha$ to be an $A$-basis element of $(\Lie U_\alpha)(A)$. Indeed, choose a subset $\Sigma$ of $\Delta$ such that the map $\Sigma \to \Delta/\Gamma$ is an isomorphism, and for each $\alpha \in \Sigma$ let $c_\alpha \in A$ be such that $c_\alpha X_\alpha \in (\Lie U_\alpha)(A)$ is an $A$-basis element. Note that $\gamma(c_\alpha X_\alpha) \in (\Lie U_{\gamma(\alpha)})(A)$ for all $\gamma \in \Gamma$ since $\gamma$ is an $A$-automorphism of $G$. By considering $\gamma_k$, we see that $\gamma(c_\alpha X_\alpha)$ has nonzero image in $(\Lie U_{\gamma(\alpha)})(k)$. Thus we may replace $X_{\gamma(\alpha)}$ by $c_\alpha X_{\gamma(\alpha)}$ for all $n \geq 0$, all $\gamma \in \Gamma$, and all $\alpha \in \Sigma$, and the result follows.
\end{proof}

\begin{proof}[Proof of Theorem~\ref{theorem:thibaud}]
	Using the Existence and Isomorphism Theorems, we may and do assume $S = \Spec \bZ$. In this case, to verify the valuative criterion of properness we must show that if $A$ is a discrete valuation ring with fraction field $K$ and $\vp \in Z^1(\Gamma, G)(A)$ and $x \in G(K)$ are such that $x \cdot \gamma(x)^{-1} = \vp(\gamma)$ for all $\gamma \in \Gamma$, then (after possibly extending $A$), there is some $y \in G(A)$ such that $y \cdot \gamma(y)^{-1} = \vp(\gamma)$ for all $\gamma$.
	
	We first show that we can find $y_0$ such that the images of $y_0 \cdot \gamma(y_0)^{-1}$ and $\vp(\gamma)$ in the adjoint quotient of $G$ are equal for all $\gamma$. In other words, we will show that if $\theta\colon \Gamma \to \Aut_{G/A}$ is the homomorphism defined by $\theta(\gamma)\coloneqq \Ad_g \circ \gamma$ and $\theta$ is $G(K)$-conjugate to the given homomorphism $\beta\colon \Gamma \to \Aut_{G/S}$ (in the sense that $\theta(\gamma) = \Ad_{y_1} \circ \gamma \circ \Ad_{y_1}^{-1}$ for some $y_1 \in G(K)$), then (after possibly extending $A$) the homomorphism $\theta$ is $G(A)$-conjugate to $\beta$. Since $\theta$ is $G(K)$-conjugate to $\beta$, it preserves a pinning of $G_K$. By Lemma~\ref{lemma:generic-pinning-gives-pinning}, it follows that $\beta$ preserves a pinning of $G$. By the conjugacy of pinnings, we may therefore assume that $\theta$ and $\beta$ preserve a common pinning. Since $\theta$ and $\beta$ also factor through a common component of $\Aut_{G/A}$ (by $G(K)$-conjugacy), it follows from \cite[Proposition 7.1.6]{Conrad:Reductive} that in this case $\theta = \beta$, as desired.
	
	Using the previous paragraph, we may now assume that $\vp(\gamma)$ is central for each $\gamma \in \Gamma$. In particular, the map $\gamma \mapsto \vp(\gamma)$ is a $1$-cocycle $g\colon\Gamma \to T$. By Lemma~\ref{lemma:chevalley-steinberg-replacement}, since $g$ is $G$-conjugate to the constant map $g \mapsto 1$, after extending $A$ we may assume there is some $n_0 \in N_G(T)(K)$ such that $n_0 \cdot \gamma(n_0)^{-1} = \vp(\gamma)$ for all $\gamma \in \Gamma$. By \cite[Corollary 15]{AdlerLansky:Root}, after extending $A$ further there is some $n_1 \in N_G(T)(K)^\Gamma$ with the same image as $n_0$ in $W$. In particular, if $t = n_0 n_1^{-1}$, then $t \in T(K)$ and $t \cdot \gamma(t)^{-1} = \vp(\gamma)$ for all $\gamma \in \Gamma$.
	
	By \cite[Expos\'e IX, Th\'eor\`eme 6.8]{SGA3.2}, the $A$-homomorphism $f\colon T \to Z^1(\Gamma, T)$ given by $t \mapsto (t\cdot\gamma(t)^{-1})_{\gamma \in \Gamma}$ factors through a faithfully flat homomorphism $T \to T'$, where $T'$ is the (closed) schematic image of $f$. By the previous paragraph we have $(\vp(\gamma))_{\gamma \in \Gamma} \in f(T(K))$, so also $(\vp(\gamma)) \in T'(A)$ and thus (by faithful flatness) after passing to an extension of $A$ there is some $t \in T(A)$ such that $t\cdot\gamma(t)^{-1} = \vp(\gamma)$ for all $\gamma \in \Gamma$, as desired.
\end{proof}

\bibliographystyle{alphaurl}
\bibliography{bib}

\end{document}